\documentclass[11pt]{amsart}
\usepackage{amsmath,amssymb,amsthm}

 2

\usepackage[colorlinks=true,linkcolor=blue,citecolor=blue,pdfpagelabels=false]{hyperref}

\newtheorem{thm}{Theorem}[section]
\newtheorem{lem}[thm]{Lemma}
\newtheorem{prop}[thm]{Proposition}
\newtheorem{cor}[thm]{Corollary}

\newcommand{\thmref}[1]{Theorem~\ref{#1}}
\newcommand{\lemref}[1]{Lemma~\ref{#1}}

\newcommand{\propref}[1]{Proposition~\ref{#1}}
\newcommand{\corref}[1]{Corollary~\ref{#1}}

\theoremstyle{remark}
\newtheorem{rmk}{Remark}[section]

\newenvironment{acknowledgements}{\bigskip\textbf{Acknowledgements.}}{}

\renewcommand{\geq}{\geqslant}
\renewcommand{\leq}{\leqslant}

\usepackage{enumerate}

\begin{document}
\title[The quantitative distribution of Hecke eigenvalues]
{The quantitative distribution of Hecke eigenvalues of Maass forms}
\author[M. Kumari, J. Sengupta]{Moni Kumari and Jyoti Sengupta}

\address{Department of Mathematics, Bar-Ilan University, Ramat Gan 52900, Israel.}
\email{monimath1990@gmail.com}

\address{
	Department of Mathematics, Indian Association for the Cultivation Of Science
	2A and 2B Raja S C Mullick Road
	Kolkata 700032, India.}
\email{sengupta@math.tifr.res.in}

\date{\today}

\subjclass[2010]{Primary 11F41; Secondary 11F30.}

\keywords{Maass forms, Fourier coefficients, distributions}
\maketitle

\begin{abstract}
	Let $f$ be a normalized Hecke-Maass cusp form of weight zero for the group $SL_2(\mathbb Z)$.
	This article presents several quantitative results about the 
	distribution of Hecke eigenvalues of $f$. Applications to the $\Omega_{\pm}$-results for the Hecke eigenvalues of $f$ and its symmetric square sym$^2(f)$ are also given.
\end{abstract}

\section{Introduction and statement of results}
Let  $f$ 
 be a normalized primitive holomorphic cusp form of even weight $k\geq 12$ for the group $SL_2(\mathbb{Z})$. The generalized Ramanujan conjecture for such $f$ was proved by Deligne in 1974 which implies that the normalized Hecke eigenvalues $\lambda_f(p)$, with $p$ running through all the primes, lie in the interval $[-2,2]$.
  The asymptotic distribution of the Hecke eigenvalues $\lambda_f(p)$, as the prime $p$ varies, is an interesting and difficult problem.
 Inspired by the Sato-Tate conjecture for elliptic curves, Serre in the 1960s conjectured, settled recently \cite{lg}, that for any such $f$, $\lambda_f(p)$ for $p \le x$ distribute nicely: the sequence $\{\lambda_f(p)\}$ is equidistributed in $[-2,2]$ with respect to the Sato-Tate measure
 $\frac{1}{2\pi}\sqrt{4-t^2}dt.$
 More precisely, given any interval $I\subseteq [-2,2]$, 
 \begin{equation}\label{st}
 \lim_{x \rightarrow \infty}\frac{\# \{p\leq x: \lambda_f(p)\in {\rm I}\}}{\pi(x)}=\frac{1}{2\pi} \int_{I}\sqrt{4-t^2}dt,
 \end{equation}
where $\pi(x)$ denotes the number of primes less than or equal to $x$ for 
 any real number $x\ge 2$. 
 This is also called the Sato–Tate conjecture. It has significant implications in number theory. 
It is well-known that the Langlands’ functoriality conjectures imply the Sato-Tate conjecture for  automorphic representations on GL(2). In absence of a proof of Langlands’ functoriality it is an interesting question, “how much” Sato-Tate one can squeeze out of the handful of known cases of functoriality. It is clear and not very deep that these cases of functoriality imply some distributional results of Hecke eigenvalues, and it leads to a challenging optimization problem to push this to the limit. This type of question was probably initiated by Serre some 20 or 30 years ago.
Regarding the large values of $\lambda_f(p)$, 
Serre \cite[Appendix]{sha} in a letter to Shahidi has proved the following. 
\begin{thm}[{Serre}]\label{ser}
 For every $\epsilon >0$  and $c=2 \cos({2\pi }/{7})\simeq 1.247,$ there are infinitely many primes $p$ such that 
 $\lambda_f(p) >c-\epsilon$
 and infinitely many primes $q$ with $\lambda_f(q) <-(c-\epsilon).$
  Moreover, these infinite sets have positive upper densities.
 \end{thm}

In the same letter he was also interested to the following related questions:
\begin{itemize}
\item[a.]\label{a}
For a given $\epsilon >0$, does there exist a non-zero lower bound, depending only on $\epsilon$, for 
the upper density of the set of primes $p$ with $\lambda_f(p) >c-\epsilon$?
\item[b.]\label{b}
Is it possible to prove an analogous result like \thmref{ser} for any $c \in (0,2]$ without assuming the Deligne's bound for the Fourier coefficients?
 \item[c.]\label{c}
 For a given $\epsilon >0$, does there exist an explicit constant $N(\epsilon)$ such that there is a prime $p\leq N(\epsilon)$ with
 $\lambda_f(p) >c-\epsilon$?
\end{itemize}
 Regarding question (b), Serre remarked that without using Deligne's result one can show that 
for every $\epsilon >0$, there are infinitely many primes $p$ such that 
$|\lambda_f(p)|> c-\epsilon$, where $c$ is as in \thmref{ser}.
It would also be worth to point out that Chiriac and Jorza in \cite{cj} studied questions (a) and (b) for the constraint $|\lambda_{f}(p)|>1$ and Walji
in \cite{w18} studied question (b) for $\lambda_{f}(p)>0.904$.

In a similar spirit, it would also be interesting to look these questions for small values of $\lambda_f(p)$.
In this direction, using Deligne's result Serre \cite[Appendix]{sha} proved that for every $\epsilon >0$, 
there are infinitely many primes $p$ such that 
$|\lambda_f(p)|< \sqrt{{2}/{3}}+\epsilon.$ 
It makes sense to point out that the above Serre's results are obsolete for holomorphic cusp forms because now the Sato-Tate conjecture is known therefore we have a very good understanding of the distribution of Hecke eigenvalues.

In contrast with holomorphic forms, it is not known that Maass forms satisfy the Ramanujan-Petersson conjecture,
although one expects it to be true because the conjecture fits correctly in the general Langlands functoriality program. It is a major unsolved problem in the general theory today and its solution would have
significant consequences in number theory and related areas.
However, in the literature there are some quantitative results towards this conjecture.
Ramakrishnan \cite{ram} proved that for a Maass form $f$, the Ramanujan conjecture is true for primes with the lower
Dirichlet density (or, analytic density) at least ${9}/{10}$. This lower Dirichlet density
is later improved to ${34}/{35}$ by Kim and Shahidi \cite{ks}. Recently, Luo and Zhou \cite{lz} have 
refined the result of Kim and Shahidi from Dirichlet density to the natural density.

As in the case of holomorphic cusp form, the Sato-Tate conjecture is expected to hold for Maass cusp forms as well.
In the same letter to Shahidi, Serre asked the analogous questions as discussed above for the Hecke eigenvalues
of Maass forms. According to him, his method has an obstruction due to lack of knowledge of the Ramanujan conjecture
(or, at least, bounded Hecke eigenvalues) for Maass forms.

The main purpose of this article is to study distribution of Hecke eigenvalues of a weight zero Hecke-Maass cusp form for the group $SL_2(\mathbb Z)$ in various intervals. The main idea that we use to prove all our distribution results is to use analytic prime number theorems and known cases of functoriality of symmetric power $L$ functions of a Hecke-Maass cusp form $f$ which lead to a polynomial optimization problem.
Our first result gives an affirmative answer to the Serre's above questions when
$|\lambda_f(p)| <1$.
 \begin{thm}\label{l1ar}
Let $f$ be a Hecke-Maass cusp form of weight zero for the group $SL_2(\mathbb Z)$. 
Assume the Fourier coefficients $\lambda_f(n)$ of $f$ satisfy the Ramanujan conjecture. Then we have
\begin{equation}
 \liminf_{x\rightarrow \infty}\frac{\#\{p\leq x: ~|\lambda_f(p)|<1\}}{\pi(x)}\geq 0.00010077.
\end{equation}
 \end{thm}
 
Unconditionally,  we prove the following.
 \begin{thm}\label{l1}
There are infinitely many primes $p$ such that $|\lambda_{f}(p)|<1$. 
\end{thm}

Now we give an upper bound for the smallest prime $p$ such that $|\lambda_{f}(p)|<1$ which answers 
 question (c) of Serre in this situation.
\begin{thm}\label{1b}
Let $f$ be a Hecke-Maass cusp form of weight zero for the group $SL_2(\mathbb Z)$ with Laplacian eigenvalue $\frac{1}{4}+u^2$.
Then there exists a prime 
\begin{equation}\label{sb}
p\ll \exp \bigg(c \log^2(1+|u|)\bigg)
\end{equation}
such that $|\lambda_{f}(p)|<1$. Here $c$ is an absolute constant and the implied constant in the $\ll$ symbol in \eqref{sb} is also absolute
though ineffective.
\end{thm}

Next we obtain a more general result about the small values of $|\lambda_f(p)|$ without assuming the Ramanujan conjecture. 
More precisely, we prove the following.
\begin{thm}\label{sra}
	Let $f$ be a Hecke-Maass cusp form of weight zero for the group $SL_2(\mathbb Z)$ and $a>1$ be a real number.
	Then 
	\begin{equation}
	\liminf_{x\rightarrow \infty}\frac{\#\{ p\leq x: ~| \lambda_f(p)|<\sqrt{a}\}}{\pi(x)}\geq \frac{a-1}{a}.
	\end{equation}
\end{thm}
\begin{rmk}
	This result has many interesting consequences. For example,  when we take $a = 4$, then it immediately
	shows that the set of primes satisfying the Ramanujan conjecture has lower
	density $> 3/4.$ Now as mentioned earlier, Luo and Zhou in \cite{lz} has shown that this set of primes
	has density at least $34/35$. Certainly $34/35 > 3/4$ , but the point is Luo and Zhou's
	proof is long whereas our result which is nontrivial has been obtained by
	very elementary means and is comparatively very short. 
	\end{rmk}

Our next results are concerned with the large values of $|\lambda_f(p)|$.
\begin{thm}\label{1to2}
Let $f$ be a Hecke-Maass cusp form of weight zero for the group $SL_2(\mathbb Z)$. 
 Then
 \begin{itemize}
  \item [a.]
 %\begin{equation*}
 ${\displaystyle{\liminf_{x\rightarrow \infty}\frac{\#\{p \leq x:~ 0.908< |\lambda_f(p)|< 1.928\}}{\pi(x)}\geq 0.0054...}}$
%\end{equation*}
\end{itemize}
and 
\begin{itemize}
\item[b.]
%\begin{equation*}
 ${\displaystyle{\liminf_{x\rightarrow \infty}\frac{\#\{p \leq x: ~1< |\lambda_f(p)|< 2\}}{\pi(x)}\geq 0.164880.}}$
%\end{equation*}
\end{itemize}
\end{thm}

\begin{rmk}
The density ``0.164880" obtained in Part (b) of \thmref{1to2} is close to its predicted density (which is 0.3910022) by \eqref {st}. Also this result
 refines  \cite[Theorem B]{cj} in two ways: first, we obtain a lower bound for the limit infimum whereas in \cite[Theorem B]{cj} they have obtained an upper Dirichlet density and second, our bound is much stronger than their bound. It is also worth pointing out that 
Walji in \cite{walji} has proved that the set of primes 
	$\{p: \lambda_{f}(p)>0.778-\epsilon\}$ has upper  Dirichlet  density  greater  than  $1/100$ for 
	any $\epsilon>0$.
	\end{rmk}
\begin{rmk}
If we apply the method of the proof of \thmref{1to2} to the polynomial $-x^2(x^2-a)(x^2-4)$ where $a=(1.189)^2$ then 
one can prove that the set of primes $p$ with $1.189< |\lambda_f(p)|< 2$ has lower density at least 0.0362. The 
reason to include this remark here is that there is a different proof of this fact given by 
Murty \cite[cor.1, p.527]{murty85} which he used to get an $\Omega$-result for the Fourier coefficients of $f$.
Murty's proof implicitly assumes the Ramanujan
conjecture for the Hecke eigenvalues of a Hecke-Maass cusp form and this
is not known, whereas this remark gives an unconditional proof.
	\end{rmk}

Next we prove a result about the phenomenon $|\lambda_f(p)|>\sqrt{2}$ which gives \corref{scs} that will be used later to get an  
$\Omega$-result about $\lambda_{sym^2 f}(n)$ (see proof of \corref{opm}).
\begin{thm}\label{sr2r}
Let $f$ be a Hecke-Maass cusp form of weight zero for the group $SL_2(\mathbb Z)$. 
 \begin{itemize}
	\item [a.]
Assume the Fourier coefficients $\lambda_f(n)$ of $f$ satisfy the Ramanujan conjecture. Then 
\begin{equation*}
 \liminf_{x\rightarrow \infty}\frac{\#\{p\leq x: ~| \lambda_f(p)|>\sqrt{2}\}}{\pi(x)}\geq \frac{1}{32}.
\end{equation*}
\item[b.]
Unconditionally, there are infinitely many primes $p$ such that  
$|\lambda_f(p)|> \sqrt{2}$.
\end{itemize}
 \end{thm}

Combining \thmref{l1} and part $(b)$ of \thmref{sr2r} with the identity
$\lambda_{sym^2 f}(p)=\lambda_{f}(p^2)=\lambda_f(p)^2-1$,
we obtain the following immediate corollary about the sign changes for the sequence
$\{\lambda_{sym^2 f}(p)\}$.
\begin{cor}\label{scs}
There are infinitely many primes $p$ such that $\lambda_{sym^2 f}(p)>0$ and also there are 
infinitely many primes $q$ such that $\lambda_{sym^2 f}(q)<0$.
\end{cor}

\begin{rmk}
	It seems rather unlikely that one can prove that the set of primes figuring
	in \thmref{l1} has positive lower density by our strategy used here, namely
	employing suitable polynomials $P(x) \in {\mathbb R} [x] $ and using the
	asymptotics of the even power moments of $\lambda_f(p)$. We are presently
	working on this and we hope to report our progress in a future article.
	\end{rmk}

For the other part of the paper, 
as a consequence of the Deligne's result, we know 
%\begin{equation*}
$|\tau(n)|\leq d(n)n^{11/2},$
%\end{equation*}
more generally, if $a_{f}(n)$ is the $n$th Fourier coefficient of a holomorphic normalized Hecke eigenform $f$ of weight $k$, then
%\begin{equation*}
$|a_f(n)|\leq d(n)n^{(k-1)/2},$
%\end{equation*}
where $d(n)$ is the divisor function.
So, it is natural to ask, `` Are the above bounds optimal as a function of $n$?"
In other words, are these inequalities sharp? 
It is elementary to show that
$d(n)=O(n^{\epsilon})$
for any $\epsilon >0$ (\cite[p. 296]{apo}).
Therefore, we have 
%\begin{equation*}
$|\tau(n)|\leq n^{11/2+\epsilon}.$
%\end{equation*}
Hardy showed that
$|\tau(n)|>n^{11/2},$
holds infinitely often.  The best result in this regard is due to Murty \cite{murty83}, 
who showed that there is an absolute and effective constant $c>0$ such that
\begin{equation}\label{ror}
|a_f(n)|>n^{(k-1)/2}\exp\left(\frac{c\log n}{\log \log n}\right)
\end{equation}
holds infinitely often. This is the best possible order because the maximal order of $d(n)$ is $\displaystyle{\exp\left(\frac{c'\log n}{\log \log n}\right)}$, due to Wigert \cite{wig}. In the literature, the phenomenon \eqref{ror} is referred to 
as an $\Omega$-result for the Fourier coefficient of a holomorphic cusp form.

Motivated from this, the final aim of this paper is to prove a 
$\Omega$-result for a general real sequence $\{a(n)\}$, that satisfy certain assumptions and as a consequence of that, we obtain the $\Omega$-results for the Hecke eigenvalues of a Hecke-Maass cusp form $f$ 
and its symmetric square sym$^2 (f)$.  

\begin{thm}\label{oor}
Let $\{a(n)\}$ be a real sequence such that
\begin{itemize}
  \item[(i)]
the function $n\mapsto a(n)$ is multiplicative,
\item[(ii)]
$\displaystyle{\sum_{p\leq x}a(p)^4\sim \alpha\pi(x)}$, for some constant $\alpha \geq 2$.
\end{itemize}
Then there exists a constant $c>0$ such that 
$$a(n)=\Omega\bigg(\displaystyle{{\rm {exp}}\bigg(\frac{c\log n}{\log \log n}\bigg)}\bigg).$$
\end{thm}

\begin{cor}\label{ofs}
Let $f$ be a Hecke-Maass cusp form of weight zero for the 
group $SL_2(\mathbb Z)$. Further assume $\lambda_f(n)$ (resp. $\lambda_{sym^2f}(n)$) are the normalized Fourier coefficients 
of $f$ (resp. sym$^2 (f)$). If 
$a(n)$ denotes either $\lambda_f(n)$ or $\lambda_{sym^2f}(n)$, then 
\begin{equation*}
a(n)=\Omega\bigg({\rm {exp}}\bigg(\frac{c\log n}{\log \log n}\bigg)\bigg)
\end{equation*} 
for some $c>0$. 
\end{cor}
 
 We will now sharpen our $\Omega$-result and prove an $\Omega _{\pm}$-result
 for both the sequences $\lambda _f (n) $ and $\lambda_{sym ^2 f }(n)$.
 This result explicitly shows large oscillations of both these sequences.
 
\begin{cor}\label{opm}
	For any Hecke-Maass cusp form $f$ as in \corref{ofs} with $a(n)$ denoting either $\lambda_f(n)$ or $\lambda_{sym^2f}(n)$, then we have
	$$a(n)=\Omega_{\pm}\bigg({\rm {exp}}\bigg(\frac{d\log n}{\log \log n}\bigg)\bigg),$$
	for some $d>0$
	\end{cor}
\begin{rmk}
 To the best of our knowledge \corref{ofs} and \corref{opm} 
are new in our context, i.e., for a Hecke-Maass cusp form.
\end{rmk}

\begin{rmk}
 Our method
of proof of the $\Omega$-result here has sym$^2(f)$ as a limiting case since we
don't know the asymptotics of the power moments of $\lambda_f(p)$ beyond the
eighth power. Hence we cannot prove the corresponding $\Omega$-result for
sym$^3(f)$ for instance.
\end{rmk}

 \subsection{Notations}
 % $\pi(x)$ denotes the number of primes less than or equal to $x$ for any real number $x\ge 2$. 
  For a real valued function $f$ and a positive function $g$, the symbol 
 ``$f(y) =O(g(y))$" or ``$f(y) \ll g(y)$"means
 that there is a constant $c>0$ such that $|f(y)|\leq cg(y)$ for any $y$ in the concerned domain.
 The dependence of this implied constant on some parameter(s) may sometimes
 be displayed by  a suffix (or suffixes).
Furthermore, the notation $f(x)\asymp g(x)$ means both the bounds $f(x)\ll
 g(x)$ and $g(x)\ll f(x)$ hold. 
 The notation ``$f(y)=o(g(y))$" means that $f(y)/g(y) \rightarrow 0$
 as $y\rightarrow \infty$ and ``$f(y) \sim g(y)$" means $f(y)-g(y) =o(g(y))$.
 The symbol 
 $f=\Omega(g)$
 means that $\limsup_{x\rightarrow \infty }\frac{|f(x)|}{|g(x)|}>0$. We write $f=\Omega_+(g)$ if
 $\limsup_{x\rightarrow \infty }\frac{f(x)}{g(x)}>0$ and $f=\Omega_-(g)$ if
 $\liminf_{x\rightarrow \infty }\frac{f(x)}{g(x)}<0$. Lastly, $f=\Omega_{\pm}(g)$ means  that
 $f=\Omega_+(g)$ and also $f=\Omega_-(g)$.
  Throughout the paper, the symbols $p$ and $q$ denote primes.
 
 \section{Preliminaries}
This  section  contains  a  very  brief  account  of  the  theory  of  Maass  forms  based  primarily  on \cite[section 1.9]{bump}.
Let $\mathbb H:=\{x+iy\in \mathbb C: y>0\}$ be the upper half-plane. Then the non-Euclidean Laplacian 
$$\Delta=-y^2\bigg(\frac{\partial^2}{\partial^2x}+\frac{\partial^2}{\partial^2y}\bigg)$$
acts on functions on $\mathbb H$. 
A Maass form $f$ of weight zero for the group $SL_2(\mathbb Z)$ is a smooth function on $\mathbb H$  such that
(i)$f(\gamma z)=f(z)$ for $\gamma \in SL_2(\mathbb Z)$; (ii) $f$ is an eigenfunction of $\Delta$; and 
(iii)$ f(x+iy)=O(y^N)$ as $y \rightarrow \infty$ for some $N$.

If $f$ is a Maass form with the Laplacian eigenvalue $\lambda$, then one 
writes $\lambda=s(1-s)$ with $s=1/2+iu$ where $s, u \in \mathbb C$, $u$
being known as the spectral parameter.
In other words, we have $\lambda=1/4+u^2$. 
%Selberg in 1965 showed that the Laplacian eigenvalue $\lambda$ cannot be less than $3/16$. 
A Maass form $f$ for $SL_2(\mathbb Z)$ with spectral parameter $u$ admits the following Fourier expansion
\begin{equation*}
f(x+iy)=\sqrt{y}\sum_{n \in \mathbb Z}a_f(n)K_{iu}(2\pi|n|y)e(nx).
\end{equation*}
Here $e(x)=e^{2\pi ix}$ and $K_{iu}$ is the Bessel function. Note that
$f$ is called a Maass cusp form if $a_0(f)=0$.

As in the case of holomorphic modular forms, there is a parallel Hecke theory for Maass forms. In particular, 
there is an orthonormal
basis of Maass cusp forms consisting of forms that are common eigenfunctions
of the Hecke operators $T_n$, $n\geq 1$. The elements in the basis are known as Hecke-Maass cusp forms.
If $f$ is a Hecke-Maass cusp form with $n$th Hecke eigenvalue $\lambda_f(n)$, then it is known that 
 $$a_f(n)=a_f(1)\lambda_f(n)~~~~~~~~~~~~~~ (n\geq 1).$$
 We call such an $f$ normalized if $a_f(1)=1$. So the Fourier coefficients of a 
 normalized Hecke-Maass cusp form are same as its Hecke eivenvalues. 
 The eigenvalues $\lambda_f(n)$ of $f$ are multiplicative
and satisfy
\begin{equation}\label{hr}
\lambda_f(n)\lambda_f(m)=\sum_{d|(m,n)}\lambda_f\left(\frac{mn}{d^2}\right).
\end{equation}
In particular, for any prime $p$, we have $\lambda_f(p^2)=\lambda_f(p)^2-1$.
The general Ramanujan conjecture asserts that for a normalized Hecke-Maass cusp form $f$ the $p$th Hecke eigenvalue $\lambda_f(p)$ satisfies
the bound $$|\lambda_f(p)|\leq 2.$$
Although this conjecture is wide open, we know from the works of Kim and Sarnak \cite{kim} that 
$$|\lambda_f(p)|\leq 2p^{\frac{7}{64}}.$$

Let $f$ be a normalized Hecke-Maass cusp form of weight zero for $SL_2(\mathbb Z)$ with $n$th Hecke eigenvalue $\lambda_f(n)$. Then the associated 
$L$-function is defined by 
\begin{equation*}
L(s,f)=\sum_{n=1}^{\infty}\frac{\lambda_f(n)}{n^{s}}=\prod_{p}\left(1-\frac{\alpha_f(p)}{p^{s}}\right)^{-1}
\left(1-\frac{\beta_f(p)}{p^{s}}\right)^{-1}
\end{equation*}
which converges absolutely for Re$(s)>1$. The local parameters $\alpha_f(p)$ and $\beta_f(p)$ are related to the 
normalized Fourier coefficients in the following way:
\begin{equation*}
\alpha_f(p)+\beta_f(p)=\lambda_f(p), ~~~~~~~~\alpha_f(p)\beta_f(p)=1.
\end{equation*}
In addition to $L(s,f)$, we have the symmetric power $L$-functions
\begin{equation*}
L(s, {\rm sym}^m f)=\prod_{p}\prod_{j=0}^{m}\left(1-\frac{\alpha_f(p)^j\beta_f(p)^{m-j}}{p^{s}}\right)^{-1},
\end{equation*}
for $m\geq 1$
and the Rankin-Selberg convolution of ${\rm sym}^m f$ and ${\rm sym}^n f$ 
\begin{equation*}
L(s, {\rm sym}^m f\times{\rm sym}^n f)=\prod_{p}\prod_{j=0}^{m}\prod_{i=0}^{n}\left(1-\frac{\alpha_f(p)^j\beta_f(p)^{m-j}\alpha_f(p)^i\beta_f(p)^{n-i}}{p^{s}}\right)^{-1}.
\end{equation*}
By the works of Kim \cite{kim} and Kim-Shahidi \cite{ks, ksh}
 the symmetric $m$th power $L$-functions for $m \leq 8$ are holomorphic and
non-vanishing for Re$(s)\geq 1$.

\section{Intermediate results}
We begin this section by recalling the following result of Brumley \cite[equ. 2.23, p 989]{bru} which is about the 
asymptotics of the even power moments of $\lambda_f(p)$. 

\begin{thm}\cite{bru}\label{bru1}
Let $f$ be a Hecke-Maass cusp form with Fourier coefficients $\lambda_f(n)$. Then 
for $1\leq j\leq 8$, we have
\begin{equation}
\sum_{p\leq x}\lambda_f(p^j)=o\bigg(\frac{x}{\log x}\bigg),
\end{equation}
as $x\rightarrow \infty.$
\end{thm}
As a consequence of this, we prove two consecutive results about the asymptotics 
of the even power moments of $\lambda_f(p)$. 

\begin{lem}\label{af}
For a Hecke-Maass cusp form $f$ of level one, we have
\begin{equation}\label{afs}
\sum_{p\leq x}\lambda_f^2(p)\sim \frac{x}{\log x}, ~~~~
\sum_{p\leq x}\lambda_f^4(p)\sim 2\frac{x}{\log x},~~~~
\sum_{p\leq x}\lambda_f^6(p)\sim 5\frac{x}{\log x}
\end{equation}
for all $x$ sufficiently large.
\end{lem}

\begin{proof}
From relation \eqref{hr}, we have
\begin{align*}
\lambda_f^2(p)&=\lambda_f(p^2)+1,\\
\lambda_f^4(p)&= \lambda_f(p^4)+3\lambda_{f}(p^2)+2,\\
\lambda_f^6(p)&= \lambda_f(p^6)+5\lambda_f(p^4)+9\lambda_{f}(p^2)+5,
\end{align*}
for a prime $p$.
The required results follow once we use the prime number theorem and \thmref{bru1} for $j=2,4,6$.
\end{proof}

\begin{lem}\label{as4}
For all sufficiently large $x$, we have
\begin{equation}\label{af1}
\sum_{p\leq x}\lambda_f^8(p)\sim \frac{14 x}{\log x},
\end{equation}
and consequently
\begin{equation*}
\sum_{p\leq x}\lambda_{sym^2 f}^4(p)\sim \frac{3x}{\log x}.
\end{equation*}
\end{lem}

\begin{proof}
From the proof of \lemref{af}, we write $\lambda_f(p^2), \lambda_f(p^4)$ and $\lambda_f(p^6)$
as a polynomial in $\lambda_f(p)$ and hence we have 
\begin{equation*}
\lambda_f(p^6)\lambda_f(p^2)=\lambda_f^8(p)-6\lambda_f^6(p)+11\lambda_f^4(p)-7\lambda_f^2(p)+1.
\end{equation*}
On the other hand, equation \eqref{hr} gives 
\begin{equation*}
\lambda_f(p^6)\lambda_f(p^2)=\lambda_f(p^8)+\lambda_f(p^6)+\lambda_f(p^4).
\end{equation*}
Combining last two equations 
\begin{equation*}
\lambda_f^8(p)=\lambda_f(p^8)+\lambda_f(p^6)+\lambda_f(p^4)+6\lambda_f^6(p)-11\lambda_f^4(p)+7\lambda_f^2(p)-1.
\end{equation*}
From \thmref{bru1} and \lemref{af},  for sufficiently large $x$, we obtain
\begin{equation*}
\sum_{p\leq x}\lambda_f^8(p)\sim \frac{14 x}{\log x},
\end{equation*}
and this proves the first part of the lemma.

For the next claim, we know
\begin{equation*}
\lambda_{sym^2 f}^4(p)=\lambda_{ f}^4(p^2)=\big(\lambda_f^2(p)-1\big)^4=
\lambda_f^8(p)-4\lambda_f^6(p)+6\lambda_f^4(p)-4\lambda_f^2(p)
+1.
\end{equation*}
Now for a large $x$, we take sum on the both sides over all the primes $p\leq x$ in the last identity and then use \lemref{af} and \eqref{af1} for the asymptotics of the involved sums to obtain the required result.
\end{proof}

Next for a Hecke-Maass cusp form $f$, we also require a lower bound for the sum $\sum_{1\leq p\leq x}|\lambda_f(p)|$. 
Using \lemref{af} for the asymptotics for the even power moments of $\lambda_f(p)$, we prove the following.
%Fouvry and Ganguly in \cite{fg} proved
\begin{prop}\label{fg1}
For a Hecke-Maass cusp form $f$ of level one we have the bound
\begin{equation}
\sum_{1\leq p\leq x}|\lambda_f(p)|\gg_f \frac{x}{\log x},
\end{equation}
for all $x$ sufficiently large. 
\end{prop}

\begin{proof}
From  the Cauchy-Schwarz inequality, we have
\begin{align*}
\bigg(\sum_{x<p\leq 2x}\lambda_f^2(p)\bigg)^2\leq
\bigg(\sum_{x<p\leq 2x}|\lambda_f(p)|\bigg)\bigg(\sum_{x<p\leq 2x}|\lambda_f^3(p)|\bigg),
\end{align*}
i.e.,
\begin{align*}
\frac{\bigg(\displaystyle{\sum_{x<p\leq 2x}}\lambda_f^2(p)\bigg)^2}{\displaystyle{\sum_{x<p\leq 2x}}|\lambda_f^3(p)|}\leq
\sum_{x<p\leq 2x}|\lambda_{f}(p)|.
\end{align*}
From \lemref{af}, we have
\begin{equation*}
\sum_{x<p\leq 2x}\lambda_f^2(p)\sim \frac{x}{\log x}.
\end{equation*}
Therefore, to complete the proof of the proposition it suffices to show that $\displaystyle{\sum_{x<p\leq 2x}|\lambda_{f}^3(p)|}\gg \frac{x}{\log x}$. For that,
first we use the Cauchy-Schwarz inequality to write
\begin{align*}
\bigg(\sum_{x<p\leq 2x}|\lambda_f^3(p)|\bigg)^2\le 
\bigg(\sum_{x<p\leq 2x}\lambda_f^2(p)\bigg)\bigg(\sum_{x<p\leq 2x}\lambda_f^4(p)\bigg)
\end{align*}
and then use \lemref{af} for the sums on the right hand side in the above inequality.
\end{proof}
	
From \thmref{bru1} and \propref{fg1}, we conclude the following corollary about the sign change of the sequence $\{\lambda_f(p)\}$, which is also known by Elliott-Kish in \cite{ek}.
This result will be useful for proving the $\Omega_{\pm}$-result; see \corref{opm}.
\begin{cor} \label{scf}
	There are infinitely many sign changes for the sequence $\{\lambda_f(p)\}$.
\end{cor}
\begin{proof}
	It suffices to show that for any real number $x$, sufficiently large, the interval $I_x=(x,2x]$ contains at 
	least two distinct primes $p$ and $q$ with $\lambda_f(p)<0$ and $\lambda_f(q)>0$.	
Here we give a proof for the existence of $p$ since the other case can be handled exactly in the same way.
So, we assume the first claim is not true, i.e., $\lambda_f(p)\geq 0$ whenever $p \in(x,2x]$.
Then from \propref{fg1}
$$\sum_{x< p\leq 2x}\lambda_f(p)\gg_f \frac{x}{\log x}$$
but taking $j=1$ in \thmref{bru1}, we know
$$\sum_{x< p\leq 2x}\lambda_f(p)=o\bigg(\frac{x}{\log x}\bigg),$$ which is not compatible with the above estimate. 
Therefore our assumption is wrong and hence it completes the proof.
\end{proof}

\section{Proof of \thmref{l1ar}}
For a prime $p$, we define $$V(p)=(g(p)+a)^2,$$
where $g(p)=-\frac{1}{14}\lambda_f^8(p)+\frac{2}{9}\lambda_f^6(p)+\frac{17}{9}\lambda_f^2(p)-2$, an even polynomial of degree
$8$ in the variable $\lambda_f(p)$ and $a$ is a positive real number, chosen later. 
Using the prime number theorem and the asymptotics for the even power moments of $\lambda_{ f}(p)$ given in \eqref{afs} and \eqref{af1}, we have 
\begin{equation}\label{ga1}
 \sum_{p\leq x}g(p)=o\left(\frac{x}{\log x}\right).
\end{equation}
Let $S:=\{p: |\lambda_f(p)| <1\}$. From our assumption, 
$|\lambda_f(p)| \leq 2$ for all primes $p$ hence
$1\leq |\lambda_f(p)|\leq2$ for  $p \notin S$ and 
 also 
 for any prime $p$ one can check that 
 $g^2(p)\leq 15.093$, so
 \begin{equation}\label{gu}
 \sum_{p\leq x}g^2(p)\leq 15.093~ \frac{x}{\log x}.
 \end{equation}
Combine \eqref{ga1} and \eqref{gu} to obtain
 $$ \sum_{p\leq x}V(p)\leq (a^2+15.093)\frac{x}{\log x}+o\left(\frac{x}{\log x}\right).$$
Positivity of $V(p)$ gives us  
\begin{equation}\label{l11}
 \sum_{p\leq x, p\notin S}V(p) \leq (a^2+15.093)\frac{x}{\log x}+o\left(\frac{x}{\log x}\right).
 \end{equation}
Now it is easy to check that for any $p\notin S$, $g(p)\geq 0.039$.  Therefore
\begin{align}\label{l22}
  \sum_{p\leq x, p\notin S}V(p)&\geq (a+0.039)^2\sum_{p\leq x, p\notin S}1\notag \\
  &=(a+0.039)^2\left(\frac{x}{\log x}-\sum_{p\leq x, p\in S}1\right).
\end{align}
From \eqref{l11} and \eqref{l22}, we have 
\begin{align*}
 {{\sum_{p\leq x, p\in S}}1}\geq \Bigg(1-\frac{a^2+15.093}{(a+0.039)^2}\Bigg)\frac{x}{\log x}+o\left(\frac{x}{\log x}\right).
\end{align*}
As a function of $a$, the rational function $1-\frac{a^2+15.093}{(a+0.039)^2}$ attains its maxima at $a=387$ with maximum value $0.00010077$.  Therefore
$$\displaystyle{\liminf_{x\rightarrow \infty}}\frac{\#\{p\leq x:~ |\lambda_f(p)|< 1 \}}{\frac{x}{\log x}}\geq 0.00010077$$ 
which completes the proof.

\section{Proof of \thmref{l1}}
In view of the relation $\lambda_f(p^2)=\lambda_f(p)^2-1$, the theorem is equivalent to the statement that there are infinitely many primes $p$ with $\lambda_{ f}(p^2)<0$. Therefore we will achieve the theorem by proving that for any $x$, sufficiently large, the interval $(x, 2x]$ contains a prime $p$ with
$\lambda_{f}(p^2)<0$. 

So, let $x$ be a sufficiently large fixed real number. By way of contradiction, assume $\lambda_{f}(p^2)\geq 0$ for all primes $p \in (x, 2x]$. 
We next claim that the positivity assumption will forced us to have 
\begin{equation*}
\frac{x}{\log x}\ll \sum_{x<p\leq 2x}\lambda_{f}(p^2).
\end{equation*}
Once we prove this claim then we are done because this will contradict  \thmref{bru1}, namely 
\begin{equation*}
\sum_{x<p\leq 2x}\lambda_f(p^2)=o\bigg(\frac{x}{\log x}\bigg).
\end{equation*}
Now to prove the claim we proceed as follows.
From  the Cauchy-Schwarz inequality, we have
\begin{align*}
\bigg(\sum_{x<p\leq 2x}\lambda_f^2(p^2)\bigg)^2\leq
\bigg(\sum_{x<p\leq 2x}\lambda_f(p^2)\bigg)\bigg(\sum_{x<p\leq 2x}\lambda_f^3(p^2)\bigg),
\end{align*}
i.e.,
\begin{align}\label{es}
\frac{\bigg(\displaystyle{\sum_{x<p\leq 2x}}\lambda_f^2(p^2)\bigg)^2}{\displaystyle{\sum_{x<p\leq 2x}}\lambda_f^3(p^2)}\leq
\sum_{x<p\leq 2x}\lambda_{f}(p^2).
\end{align}

In order to obtain a lower bound of $\displaystyle{\sum_{x<p\leq 2x}\lambda_{f}(p^2)}$, we need to get an upper bound for
$\displaystyle{\sum_{x<p\leq 2x}\lambda_f^3(p^2)}$ and a lower bound for $\displaystyle{\sum_{x<p\leq 2x}\lambda_f^2(p^2)}$. For that
\begin{align*}
\sum_{x<p\leq 2x}\lambda_f^2(p^2)=\sum_{x<p\leq 2x}\big(\lambda_f^2(p)-1\big)^2=\sum_{x<p\leq 2x}\lambda_f^4(p)+\sum_{x<p\leq 2x}1-2\sum_{x<p\leq 2x}\lambda_f^2(p)
\end{align*}
From \lemref{af} and the prime number theorem, we have
\begin{equation}\label{esa}
\sum_{x<p\leq 2x}\lambda_f^2(p^2)\sim \frac{x}{\log x}.
\end{equation}
Similarly, we obtain
\begin{equation}\label{esa1}
\sum_{x<p\leq 2x}\lambda_f^3(p^2)\sim \frac{x}{\log x}.
\end{equation}
Thus combining \eqref{es}, \eqref{esa} and \eqref{esa1} we have
\begin{equation*}\label{esl}
\frac{x}{\log x}\ll \sum_{x<p\leq 2x}\lambda_{f}(p^2)
\end{equation*}
which proved the claim.

\section{proof of \thmref{1b}}
As in the proof of \thmref{l1}, we address this problem for the smallest prime $p$ such that $\lambda_{sym^2 f}(p)=\lambda_f(p^2)<0$. For that consider the function 
\begin{equation*}
\psi(sym^2 f,x):= \sum_{n\leq x}\Lambda _{sym^2 f}(n),
\end{equation*}
where $\Lambda _{sym^2 f}(n)$ is the $n$th coefficient of the negative of the logarithmic derivative of the $L$-function $L(s, {\rm sym}^2f)$ which are supported only on prime powers.
Now since the Rankin-Selberg $L$-function $L(s, {\rm sym}^2 f\times{\rm sym}^2 f)$ exists and has a simple pole at $s=1$ (see, \cite[p. 180]{hl}), then 
from \cite[p. 110, exercise 6]{ik} we have
\begin{equation}\label{sub}
\psi(sym^2 f,x)=\sum_{p\leq x}\lambda_{sym^2 f}(p)\log p+O\big(\sqrt{x}d^2\log^2 (xq(sym^2 f))\big),
\end{equation}
where the implied constant is absolute and $d$ is the degree of the $L$-function  $L(s, {\rm sym}^2f)$ (notice that $d=3$).
The $L$-function  $L(s, {\rm sym}^2f)$ is entire, non-vanishing at $s = 1$, and does not have a Siegel zero; see \cite{hl}. Using this, the prime number theorem for $L(s, {\rm sym}^2f)$ becomes (see, \cite[equ. 5.52]{ik}), 
\begin{equation}\label{pns}
\psi(sym^2 f,x)=O\left(x\sqrt{q(sym^2 f)} e^{-\frac{c}{162}\sqrt{\log x}}\right),
\end{equation} 
where $c$ is an absolute constant appearing in the proof of Theorem $5.10$ of \cite{ik}.
%\noindent
After substituting these approximations in \eqref{sub}

\begin{equation*}
\sum_{p\leq x}\lambda_{sym^2 f}(p)\log p=O\big(\sqrt{x}d^2\log^2 (xq(sym^2 f))\big)+
O\left(x\sqrt{q(sym^2 f)} e^{-\frac{c}{162}\sqrt{\log x}}\right),
\end{equation*}
i.e., $$\sum_{p\leq x}\lambda_{sym^2 f}(p)\log p \ll x\sqrt{q(sym^2 f)} e^{-\frac{c}{162}\sqrt{\log x}}.$$
We now apply the partial summation formula to the left hand side of the above inequality to obtain
\begin{equation}\label{sub1}
\sum_{p\leq x}\lambda_{sym^2 f}(p)\ll \frac{x}{\log x}\sqrt{q(sym^2 f)} e^{-\frac{c}{162}\sqrt{\log x}}.
\end{equation} 

Let $p_0$ be the smallest prime such that $\lambda_f(p_0^2 ) < 0$. Now given
any $A>0$ if $p_0 \leq A$ , then we are done since $A$ can be absorbed in
the implied constant figuring in the bound occurring in the theorem. Hence
we can assume that $p_0$ is as large as we please. 
Choose $y =p_0 - \epsilon$ where $0< \epsilon < 1/2$. Then
our assumption on $p_0$ implies that $\lambda_f(p^2) \geq 0$ for all $p \leq y$ 
and we can choose $y$ as large as we please.
Then using the same argument as for \eqref{esl}, we obtain
\begin{equation}\label{slb}
\frac{y}{\log y}\ll \sum_{p\leq y}\lambda_{sym^2 f}(p).
\end{equation}
Note that inequality \eqref{sub1} is true for any sufficiently large real number $x$, in particular for $y$ also. Now
comparing the upper \eqref{sub1} and lower \eqref{slb} bounds for the sum $\sum_{p\leq y}\lambda_{sym^2 f}(p)$ and then substituting $q(sym^2 f)=(1+|u|)^2$ gives the result.
%\end{proof}
\section{Proof of \thmref{sra}}
For a given $a>1$, we consider the polynomial $V_a(p)=\lambda_f^2(p)-a$ and let $S_a=\{p: |\lambda_f(p)|<\sqrt{a}\}$. Using the fact that $\lambda^2_f(p)\geq a$ for $p\notin S_a$, we have 
	\begin{align*}
	0\leq \sum_{p\leq x, p\notin S_a}V_a(p)&=\sum_{p\leq x}V_a(p)-\sum_{p\leq x, p\in S_a}V_a(p)\\
	&=\frac{x}{\log x}-a\frac{x}{\log x}+o\Big(\frac{x}{\log x}\Big)-\sum_{p\leq x, p\in S_a}\lambda^2_f(p)+a\sum_{p\leq x, p \in S_a}1.
	\end{align*}
	In other words, we have
	\begin{align*}
	\sum_{p\leq x, p\in S_a}\lambda^2_f(p)+(a-1)\frac{x}{\log x}\leq a\sum_{p\leq x, p \in S_a}1+o\Big(\frac{x}{\log x}\Big).
	\end{align*}
	Since $\lambda_f^2(p)\geq 0$ for any prime $p$ and $a>1$, hence we can ignore the sum $\sum_{p\leq x, p\in S_a}\lambda^2_f(p)$ from the left hand side of the above inequality and then divide by $\frac{x}{\log x}$ on both the sides to get 
	$$\frac{a-1}{a}\leq \liminf_{x\rightarrow \infty}\displaystyle{\frac{\sum_{p\leq x, p \in S_a}1}{\pi(x)}},$$
	which completes the proof.
%\end{proof}

\section{Proof of \thmref{1to2}}
 \begin{proof}[a.]
 Consider the polynomial
 $$g(t)=-\frac{1}{5.7}t^6+\frac{4}{8.5}t^4+\frac{17}{18}t^2-1.$$
  Notice that $g(t)$
   is $<0$ for $t\in [0, t_0)\cup (t_1, \infty)$ 
   and  $>0$ for $t \in (t_0,t_1)$ where $t_0=0.908$ and $t_1=1.928$ and also $g(t_0)=0=g(t_1)$. 
   In the polynomial $g(t)$, we put $t=\lambda_f(p)$ and consider the sum $ \sum_{p \leq x}g(\lambda_f(p))$. Using \lemref{af} and \lemref{as4} we obtain the following asymptotic,
   \begin{equation}\label{ga}
   \sum_{p \leq x}g(\lambda_f(p))\sim 0.0085 ~\pi(x)~~ {\rm as} ~~x\rightarrow \infty.
  \end{equation}
 Next we split the sum $\sum_{p \leq x}g(\lambda_f(p))$ into two parts, depending on $g$ is positive or negative
\begin{align*}
 \sum_{p \leq x}g(\lambda_f(p))=&\sum_{ p \in S_0(x)\cup S_1(x)}g(\lambda_f(p))+\sum_{ p\in S(x)}g(\lambda_f(p))
\end{align*}
where 
$S_0(x)=\{p\leq x: 0\leq |\lambda_f(p)|\leq t_0\}$, $S(x)=\{p\leq x:  t_0\leq |\lambda_f(p)|\leq t_1\}$ and
$S_1(x)=\{p\leq x: |\lambda_f(p)|> t_1\}.$
Since for $p \in S_0(x)\cup S_1(x),~ g(\lambda_f(p))<0$, so in order to get an upper bound for the sum $\sum_{p \leq x}g(\lambda_f(p))$ 
one can completely ignore the first term from the above equation. Thus
 \begin{align*}
 \sum_{p \leq x}g(\lambda_f(p))&< \sum_{ p\in S(x)}g(\lambda_f(p))
  \leq 1.561  \sum_{p \in S(x)}1.
\end{align*}
Here $1.561 $ is the maximum of $g(\lambda_f(p))$ for $p \in S(x).$
We now use  asymptotic \eqref{ga} for 
$\sum_{p \leq x}g(\lambda_f(p))$ in the last inequality  
\begin{equation*}
 \frac{0.0085}{1.561}\leq \liminf_{x\rightarrow \infty}\frac{\#\{p \leq x: 0.908\leq |\lambda_f(p)|\leq 1.928\}}{\pi(x)},
\end{equation*}
which completes the proof.
\end{proof}

\begin{proof}[b.]
 Consider the polynomial
 \begin{align*}
  g_{\alpha}(t)=&t^2(-t^2+1)(t^2-4).
 \end{align*}
  Note that $g_{\alpha}(t)$ is an even polynomial of degree 6 whose positive roots are 0,1 and 2. Also,  
 $ g_{\alpha}(t)< 0$  for $t\in [0,1)\cup (2,\infty)$ and
 %\item
 $g_{\alpha}(t)>0$ for $t\in (1,2)$.

Now consider $g_{\alpha}$ as a polynomial in $\lambda_f(p)$.  Then \lemref{af} yields
\begin{equation}\label{aa}
\sum_{p \leq x}g_{\alpha}(\lambda_f(p))\sim \pi(x)~~ {\rm as} ~~x\rightarrow \infty.
\end{equation}
%\end{enumerate}
On the other hand, if we split the sum $\sum_{p \leq x}g_{\alpha}(\lambda_f(p))$ into two parts, depending on the sign of $g_{\alpha}$, then
$$\sum_{p \leq x}g_{\alpha}(\lambda_f(p))< {M} ~\#\{p\leq x: ~1<|\lambda_f(p)|<2\},$$
where $M:={\rm {Max}}\{g_{\alpha}(\lambda_f(p)): |\lambda_f(p)|\in (1,2)\}=6.065.$ This together with 
asymptotic \eqref{aa} gives
 \begin{equation}\label{bp}
  \frac{1}{6.065}\leq \liminf_{x\rightarrow \infty}\frac{\#\{p \leq x:~ 1< |\lambda_f(p)|< 2\}}{\pi(x)},
 \end{equation}
 and hence we complete the proof.
%Notice that $\frac{(1-\alpha)}{M}$ will be maximum as $\alpha \rightarrow 0$. For  $\alpha=.0000001$, we have $ M=6.065$ and $\frac{(1-\alpha)}{M}=.164880$. So substituting this value of
%$\frac{(1-\alpha)}{M}$
%in \eqref{bp}, we obtain the desired result.
%\end{enumerate} 
\end{proof}

\section{Proof of \thmref{sr2r}}
 \begin{proof}[a.]
Let $A=\{p:|\lambda_f(p)|>\sqrt{2}\}$. Then we
use \lemref{af} to say
\begin{align*}
0\leq -\sum_{p\leq x, p\notin A}\lambda^6_f(p)+2\sum_{p\leq x, p\notin A}\lambda^4_f(p)
&=
-5\frac{x}{\log x}+\sum_{p\leq x, p\in A}\lambda^6_f(p)\\
&+4\frac{x}{\log x}-2\sum_{p\leq x, p\in A}\lambda^4_f(p)+o\left(\frac{x}{\log x}\right).
\end{align*}
In other words,
\begin{equation*}
\frac{x}{\log x}\leq \sum_{p\leq x, p\in A}\lambda^4_f(p)\big(\lambda_f^2(p)-2\big)+o\left(\frac{x}{\log x}\right).
\end{equation*}
By our assumption, for any prime $p$, $|\lambda_f(p)| \leq 2$ therefore we have
\begin{equation*}
\frac{x}{\log x}\leq 32\sum_{p\leq x, p\in A} 1+o\left(\frac{x}{\log x}\right).
\end{equation*}
Now first divide by $\pi(x)$ on both the sides of the above inequality and then taking $\displaystyle{\liminf_{x \rightarrow\infty}}$, we obtain
\begin{equation*}
\frac{1}{32}\leq \liminf_{x \rightarrow\infty} \frac{\sum_{p\leq x, p\in A} 1}{\pi(x)},
\end{equation*}
which completes the proof. 
\end{proof}

 \begin{proof}[b.]
 For this we consider the following even polynomial
 \begin{equation*}
  V(p)=-\lambda_f(p)^6+2\lambda_f(p)^4+1.
 \end{equation*}
%Notice that the polynomial $V(p)$ does not depend on the sign of $\lambda_f(p)$. 
 As in the proof of \thmref{l1}, we prove the theorem by showing that for every $x$, sufficiently large, the interval $(x,2x]$ 
 contains at least one prime $p$ such that
 $|\lambda_f(p)|>\sqrt{2}$.
By way of contradiction, assume 
$|\lambda_f(p)|\leq \sqrt{2}$ whenever $p\in (x,2x]$. This shows that for $p \in (x,2x]$,
$V(p)\geq 1$ and hence we have
\begin{equation}\label{lbt}
 \sum_{x<p\leq 2x}V(p)\geq \frac{x}{\log x}.
\end{equation}
But from \lemref{af}, we know 
\begin{equation*}
 \sum_{p\leq x}V(p)=o\left(\frac{x}{\log x}\right)~~{\rm as}~~x\rightarrow \infty
\end{equation*}
which contradicts \eqref{lbt}.
Therefore our assumption, that $|\lambda_f(p)|\leq \sqrt{2}$ whenever $p\in (x,2x]$ is false and hence this completes the proof.
\end{proof}

\section{Proof of \thmref{oor}}
We achieve this result by showing that for each $x$ sufficiently large, the interval $(x, 2x]$ gives a square-free integer $N$ such that 
\begin{equation*}
|a(N)|\geq A~{\rm {exp}}\bigg(\frac{c\log N}{\log \log N}\bigg), ~~{\rm ~for ~some~constant}~ A>0.
\end{equation*}
For a sufficiently large $x$, consider the interval $I_x:=(x,2x]$. Suppose for all primes $p\in I_x$
\begin{equation*}
a(p)=O\bigg({\rm {exp}}\bigg(\frac{c_0\log p}{\log \log p}\bigg)\bigg), ~~{\rm{for~ some  ~constant~}} 
c_0>0,
\end{equation*} 
otherwise our claim is true.
Now, fix $\delta$ such that $1<\delta <2^{\frac{1}{4}}$, i.e., $1<\delta^4<2$.
Therefore
\begin{equation}\label{ups4}
\sum_{p\in I_x, |a(p)|<\delta}a(p)^4\leq \delta^4 \frac{x}{\log x}.
\end{equation}
Combining this with assumption (ii) in the statement gives
\begin{align}\label{lbs4}
\sum_{p\in I_x, |a(p)|\geq \delta}a(p)^4&=\sum_{p\in I_x}a(p)^4-\sum_{p\in I_x, |a(p)|< \delta}a(p)^4\notag \\
&\geq (\alpha-\delta^4)\frac{x}{\log x}\notag\\
&\gg \frac{x}{\log x}.
\end{align}
Since $\alpha$ is greater than or equals to 2 therefore 
the constant that appears in the right hand side of the above inequality is non-zero.
Let $T:=\#\{p\in I_x: |a(p)|\geq \delta\}$. Then
from \eqref{lbs4}, we have
\begin{equation}\label{tb}
T\geq \frac{c_1 x}{\log x}{\rm{exp}}\bigg(\frac{-4c_0\log x}{\log \log x}\bigg),
\end{equation}
where $c_1$ is a positive constant.
Now, put $N=\prod_{p\in I_x, |a(p)|\geq \delta}p$.
Because the coefficients $a(n)$ are multiplicative, we have
\begin{equation*}
|a(N)|=\prod_{p|N}|a(p)|\geq \delta^{T}={\rm exp}\big(T\log \delta\big).
\end{equation*}
Now
\begin{equation*}
T \log x\leq \log N=\sum _{x<p\leq 2x, |a(p)|\geq \delta}\log p\leq T\log 2x,
\end{equation*}
i.e., $\log N\asymp T \log x$. Thus
\begin{equation*}
|a(N)|\geq {\rm exp}\bigg((\log \delta)\frac{\log N}{\log x}\bigg).
\end{equation*}
We now want to analyse $\log x$ as a function of $N$.
 From the prime number theorem, we have $\log N\ll x$ whereas from \eqref{tb} 
\begin{equation*}
\log N\geq T\log x\geq c_o x ~{\rm{exp}}\bigg(\frac{-4c_o\log x}{\log \log x}\bigg).
\end{equation*} 
Putting all together, we obtain
\begin{equation*}
\log x\gg \log {\log N}\geq \log c_o+\log x+O\bigg(\frac{\log x}{\log \log x}\bigg).
\end{equation*}
In other words, $\log x\asymp \log \log N$, so that 
\begin{equation*}
|a(N)|\gg {\rm exp}\bigg((\log \delta)\frac{\log N}{\log \log N}\bigg).
\end{equation*}
This completes the theorem.

\section{Proof of \corref{ofs} and  \corref{opm}}
Since both the sequences $\{\lambda_{ f}(n)\}$ and $\{\lambda_{sym^2 f}(n)\}$ are multiplicative and from
	\lemref{af} and \lemref{as4} it is also clear that they satisfy the second assumption of \thmref{oor} with $\alpha=2$ and $3$, respectively. Hence from \thmref{oor} we complete the proof of \corref{ofs}.

The proof of \corref{opm} goes as follows:
From \corref{ofs}, we know 
\begin{equation}\label{op}
a(m)>c_1{\rm {exp}}\bigg(\frac{c\log m}{\log \log m}\bigg)
\end{equation}
or
\begin{equation}
a(m)<-c_2{\rm {exp}}\bigg(\frac{c\log m}{\log \log m}\bigg)
\end{equation}
 holds infinitely often for square-free integers $m$ with some positive constants $c_1$ and $c_2$. 
Without loss of generality, assume \eqref{op} is true. So it is sufficient to prove that there is a sequence of positive integers $\{n_m\}_{m\in \mathbb N}$ and a positive constant $M$ such that
\begin{equation*}\label{om}
-a(n_m)>M{\rm {exp}}\bigg(\frac{c\log n_m}{\log \log n_m}\bigg)
\end{equation*}
for all $m\in {\mathbb N}$. Let $q$ be a prime with $a(q)<0$, such $q$ exists by 
\corref{scf} or \corref{scs} depending on $f$ or sym$^2 f$, respectively.
Then for any $m$ satisfying inequality \ref{op}, we define
$n_m= qm$ if $(q, m) = 1$ and $n_m=\frac{m}{q}$ otherwise. In the first scenario, since $(q,m)=1$ we have
$$-a(n_m)=(-a(q))a(m)\gg_q{\rm {exp}}\bigg(\frac{c\log n_m}{\log \log n_m}\bigg),$$
otherwise, we know $(n_m,q)=1$ as $m$ is square-free, hence
$$-a(n_m)=\bigg(\frac{-1}{a(q)}\bigg)a(m)\gg_q{\rm {exp}}\bigg(\frac{c\log n_m}{\log \log n_m}\bigg).$$
This completes the proof.

\begin{acknowledgements}
	The authors would like to thank Prof. Satadal Ganguly for many helpful conversations. 
	Also, it is our pleasure to 
	thank Prof. Ram Murty for sketching 
	a slightly
	different proof of  \corref{ofs} in the case of Maass cusp form through private communication.
	The authors thank the anonymous referees for a careful reading of the manuscript and they are grateful for several suggestions from the referees that led to a substantial improvement in the quality of this article.
	 The authors also thank
	the School of Mathematics, TIFR Mumbai, where this project was started when the first named author was a postdoctoral fellow there, for providing excellent working conditions.
	
	The research of the first author was supported by Israeli Science Foundation grant 1400/19.
	\end{acknowledgements}

%{\bf{Data availability.}} Data sharing not applicable to this article as no datasets were generated or analysed during the current study.

\end{document}